\documentclass[11pt]{amsart}

\usepackage{amsmath}  
\usepackage{amssymb}
\usepackage{amsthm}

\usepackage{graphicx}
\usepackage{psfrag}

\newtheorem{lemma}{Lemma}
\newtheorem{proposition}{Proposition}
\newtheorem{remark}{Remark}
\newtheorem{definition}{Definition}

\newcommand{\R}{\mathbb{R}}
\newcommand{\Id}{\mathrm{Id}}

\begin{document}

\title[Optimal support recovery with multi-penalty regularization]{Conditions on optimal support recovery in unmixing problems by means of multi-penalty regularization}

\author{Markus Grasmair}
\address{Department of Mathematical Sciences, Norwegian University of Science and Technology, 7491 Trondheim}
\email{markus.grasmair@math.ntnu.no}

\author{Valeriya Naumova}
\address{Simula Research Laboratory AS, Martin Linges vei 25, 1364 Fornebu, Oslo}
\email{valeriya@simula.no}

\maketitle

\begin{abstract}
Inspired by several real-life applications in audio processing and medical image analysis, where the quantity of interest is generated by several sources to be accurately modeled and separated, as well as by recent advances in  regularization theory and optimization, we study the conditions on optimal support recovery in inverse problems of unmixing type by means of multi-penalty regularization.  

We consider and analyze a regularization functional composed of a data-fidelity term, where signal and noise are additively mixed, a non-smooth, convex, sparsity promoting term, and a quadratic penalty term to model the noise. We prove not only that  the well-established theory for sparse recovery in the single parameter case can be translated to the multi-penalty settings, but  we also demonstrate the enhanced properties of multi-penalty regularization in terms of support identification compared to sole $\ell^1$-minimization.  We additionally confirm and support the theoretical results by extensive numerical simulations, which give a statistics of robustness of the multi-penalty regularization scheme with respect to the single-parameter counterpart. 
Eventually, we confirm a significant improvement in performance compared to standard $\ell^1$-regularization
for  compressive sensing problems considered in our experiments. 
\end{abstract}

\section{Introduction}

In many real-life applications such as audio processing or medical image analysis, one encounters the situation when given observations (most likely noisy) have been generated by several sources $u_i$ that one wishes to reconstruct separately.  
In this case, the reconstruction problem can be understood as an inverse problem of unmixing type, where the solution $u^\dag$ consists of several (two or more) components of different nature, which have to be identified and separated.  

In mathematical terms, an unmixing problem can be stated as the solution of an equation
\[
	A u^\dag = y,
\]
 where  $u^\dag = \sum_{i=1}^L u_i, u_i \in V_i$ and $V_1+\ldots+V_L = \R^N$ but $\langle V_i, V_j \rangle \simeq 0$ for $i \neq j$ in the sense that 
\begin{equation}
	\label{UoS}
		| \langle v_i, v_j \rangle | \simeq \delta_{ij} 
\end{equation}
for all $v_i \in V_i$ and $v_j \in V_j,$ $\| v_i\|_{\ell_2} = \| v_j \|_{\ell_2} = 1, i \neq j.$ 
In general, we are interested to acquire the minimal amount of information on $u$ so that we can selectively reconstruct with the best accuracy one of the components $u_{\hat \imath}$, but not necessarily also the other components $u_j$ for $j \neq \hat \imath$. In this settings, we further assume that $A$ cannot be specifically tuned to recover $u_{\hat \imath}$ but should be suited to gain universal information to recover $u_{\hat \imath}$ by a specifically tuned decoder. 

A concrete example of this setting is the \emph{noise folding phenomenon} arising in compressed sensing, related to noise in the signal that is eventually amplified by the measurement procedure. In this setting, it is reasonable to consider a model problem of the type
\begin{equation}
	\label{model_problem}
		A (u+v) = y,
\end{equation}
where $v$ is the random Gaussian noise with variance $\sigma_v$ on the original signal $u \in \R^N$ and $A \in \R^{m \times N}$ is the linear measurement matrix. Several recent works (see, for instance, \cite{Arias-castro11yeldar} and the references therein) illustrate how the measurement process actually causes the noise folding phenomenon. To be more specific, one can show that  (\ref{model_problem}) is equivalent to solving 
\begin{equation}
	\label{model_problem_eq}
		\hat A u+ \omega = y,
\end{equation}
where $\omega$ is composed by i.i.d.\ Gaussian entries with distribution $\mathbf N(0, \sigma_\omega)$, and 
the variance $\sigma_\omega$ is related to the variance of the original signal
by $\sigma_\omega^2 = \frac{N}{m} \sigma_v^2$.
This implies that the  variance of the noise on the original signal is amplified by a factor of $N/m$.

Under the assumption that $A$ satisfies the so-called restricted isometry property,
it is known from the work on the Dantzig selector in \cite{CandesTao} that one can reconstruct $u^\dagger$
from measurements $y$ as in (\ref{model_problem_eq}) such that
\begin{equation}
	\label{recon_rate}
	 \| u - u^\dagger \|_2^2 \leq C^2  2\Bigl((1+k) \frac{N}{m} \sigma^2_v\Bigr)\log N ,
\end{equation}
where $k$ denotes the number of nonzero elements of the solution $u.$ The estimate (\ref{recon_rate}) is considered (folklore) nearly-optimal in the sense that no other method can really improve the asymptotic error $\mathcal O(\frac{N}{m}\sigma_v^2).$ Therefore, the noise folding phenomenon may in practice significantly reduce the potential advantages of compressed sensing in terms of the trade-off between robustness and efficient compression (given by the factor $N/m$ here) compared to other more traditional subsampling methods \cite{6204356}.

In \cite{arfopeXX} the authors present a two-step numerical method which allows not only to recover the large entries of the original signal $u$ accurately, but also has enhanced properties in terms of support identification over simple $\ell^1$-minimization based algorithms. In particular, because of the lack of separation between noise and reconstructed signal components, the latter ones can easily fail to recover the support when the support is not given a priori. However, the computational cost of the second phase of the procedure presented in  \cite{arfopeXX}, being a non-smooth and non-convex optimization problem, is too demanding to be performed on problems with realistic dimensionalities. 
It was also shown that other methods based on a different penalization of the signal and noise can lead to higher support detection rate.  

The follow up work \cite{naumpeter}, which addresses the noise folding scenario  by means of multi-penalty regularization, provides the first numerical evidence of the superior performance of multi-penalty regularization compared to its single parameter counterparts for problem (\ref{model_problem}). In particular, the authors consider the functional
\begin{equation}
	\label{MP}
J_{p,q}(u,v) : = \| A(u+v)-y \|^2_{2} +  \alpha \| u\|_{p}^{p} + \left (\beta  \| v\|_{q}^{q} + \varepsilon \| v\|_{2}^2 \right ),
\end{equation}
here  $\alpha, \beta, \varepsilon \in \R_+$ may all be considered as regularization parameters of the problem. 
The parameter $\varepsilon>0 $ ensures the $\ell_2-$coercivity of $J_{p,q}(u,\cdot)$ also with respect to the component $v$. In the infinite dimensional setting the authors presented a numerical approach to the minimization of (\ref{MP}) for $0 \leq p <2, 2 \leq q< \infty$, based on simple iterative thresholding steps, and analyzed its convergence.

The results presented in this paper are very much inspired not only by the above-mentioned works in the signal processing and compressed sensing fields, but also by  theoretical developments in sparsity-based regularization  (see \cite{CPA:CPA20350} and references therein) and multi-penalty regularization (\cite{LP_NumMath, NP13}, just to mention a few). 
While the latter two directions are considered separately in most of the literature, there have also been some efforts to understand regularization and convergence behavior for multiple parameters and functionals, especially for image analysis \cite{BH14, SL13}.
However, to the best of our knowledge, the present paper is the first one providing a theoretical analysis  of the multi-penalty regularization with a non-smooth sparsity promoting regularization term, and an explicit comparison  with the single-parameter counterpart.
 
 \subsection{Content of the paper}
In Section 2 we concisely recall the pertinent features and concepts of multi-penalty and single-penalty regularization. We further show that $\ell^1$-regularization can be considered as the limiting case of the multi-penalty one, and thus the theory of $\ell^1$-regularization can be applied to multi-penalty setting. 
In Section 3 we recall and discuss conditions for exact  support recovery in the single-parameter case. 
The main contributions of the paper are presented in Sections 4 and 5, where we extend and generalize the results from the previous sections to the multi-penalty setting. 
In Section 5 we also open the discussion on the set of admissible parameters for the exact support recovery for unmixing problem in single-parameter as well as multi-penalty cases. In particular, we  study the sensitivity of the multi-penalty scheme with respect to the parameter choice. The theoretical findings and discussion are illustrated and  supported by extensive numerical validation tests presented in Section 6. Finally, in Section 7 we compare the performance  of the multi-penalty regularization and its single-parameter counterpart for  compressive sensing problems.

\section{Multi-Penalty and Single-Penalty Regularization}

We first provide a short reminder and collect some definitions of the standard notation used in this paper.   

The true solution $u^\dagger$ of the unmixing problem~\eqref{model_problem}
is called $k$-sparse if it has at most $k$ non-zero entries, 
i.e., $\#I = \# \mathop{{\rm supp}}(u^\dagger) \leq k$, 
where $I := \mathop{{\rm supp}}(u^\dagger) := \bigl\{i: u^\dagger_i \neq 0\}$ denotes the support of $u^\dagger$.

We propose to solve the
unmixing problem~\eqref{model_problem} with $k$-sparse true solution
using multi-penalty Tikhonov regularization
of the form
\begin{equation}\label{eq:multi}
T_{\alpha,\beta}(u,v) := \frac{1}{2}\|A(u+v)-y\|_2^2 + \alpha\|u\|_1 + \frac{\beta}{2}\|v\|_2^2 \to \min_{u,v},
\end{equation}
the solution of which we will denote by $(u_{\alpha,\beta},v_{\alpha,\beta})$.
We note that we can, formally, interpret standard $\ell^1$-regularization
as the limiting case $\beta = \infty$, setting
\[
T_{\alpha,\infty}(u,v) := 
\begin{cases}
  \frac{1}{2}\|Au-y\|_2^2 + \alpha\|u\|_1 & \text{if } v = 0,\\
  +\infty & \text{if } v\neq 0.
\end{cases}
\]
Obviously, the pair of minimizers of $T_{\alpha,\infty}$ will always
be equal to $(u_\alpha,0)$, where $u_\alpha$
minimizes $\frac{1}{2}\|Au-y\|_2^2+\alpha\|u\|_1$.

\begin{definition}
  Let $\beta \in \R_+\cup\{\infty\}$ be fixed.
  We say that a set $\mathcal{S} \subset \R^N \times \R^N$ is a
  \emph{set of exact support recovery for the unmixing
    problem with operator $A$}, if there exists $\alpha > 0$,
  such that $\mathop{{\rm supp}}(u_{\alpha,\beta}) = \mathop{{\rm
      supp}}(u^\dagger)$
  whenever the given data $y$ has the form $y = A(u^\dagger+v)$
  with $(u^\dagger,v) \in \mathcal{S}$.

  The parameters $\alpha > 0$ for which this property holds
  are called \emph{admissible for $\mathcal{S}$}.
\end{definition}

Specifically, we will study for $c>d>0$ the sets 
\begin{equation}\label{eq:ScdI}
\mathcal{S}_{c,d,I} := \bigl\{(u,v) \in \R^N \times \R^N :
\mathop{{\rm supp}}(u) = I,\
\inf_{i\in I} |u_i| > c,\
\|v\|_\infty < d\bigr\}
\end{equation}
and the corresponding class 
\[
\mathcal{S}_{c,d,k} := \bigcup_{\#I \le k} \mathcal{S}_{c,d,I}.
\]
The set $\mathcal{S}_{c,d,I}$ is a set of exact support
recovery, if there exists some regularization parameter $\alpha > 0$,
such that we can apply multi-penalty regularization
with parameters $\alpha$ and $\beta$ (or single-parameter $\ell^1$-regularization
with parameter $\alpha$ in case $\beta = \infty$)
to the unmixing problem~\eqref{model_problem} and obtain a result
with the correct support, provided that the $\ell^\infty$-norm of
the noise is smaller than $d$ and the non-zero coefficients
of $u^\dagger$ are larger than $c$.
Typical examples of real-life signals that can be modeled by signals from the set $\mathcal{S}_{c,d,I}$ can be found in Asteroseismology, see for instance \cite{arfopeXX}.

We note that this class of signals is very similar to the one studied
in~\cite{arfopeXX}. The main difference is that we focus on the case
where the noise $v$ is bounded only componentwise (that is, with respect
to the $\ell^\infty$-norm), whereas~\cite{arfopeXX} deals with
noise that has a bounded $\ell^p$-norm for some $1 \le p \le 2$.
Additionally, we allow the noise also to mix with the signal $u^\dagger$
to be identified in the sense that the supports of $v$ and $u^\dagger$
may have a non-empty intersection.
In contrast, the signal and the noise are assumed to be strictly
separated in~\cite{arfopeXX}.
\medskip

Throughout the paper we will several times refer to the
sign function $\mathop{{\rm sgn}}$, which we always
interpret as being the set valued function $\mathop{{\rm sgn}}(t) = 1$ if $t > 0$,
$\mathop{{\rm sgn}}(t) = -1$ if $t < 0$, and
$\mathop{{\rm sgn}}(t) = [-1,1]$ if $t = 0$, applied componentwise
to the entries of the vector $u_\alpha$.

We use the notation $A_I$ to denote the restriction of the operator $A$
to the span of the support of $u^\dagger$.
Additionally, we denote by
\[
J := \bigl\{i: u^\dagger_i = 0\}
\]
the complement of $I$, and by $A_J$ the restriction
of $A$ to the span of $J$.
We note that the adjoints $A_I^*$ and $A_J^*$ are simply
the compositions of the adjoint $A^*$ of $A$ with the projections
onto the spans of $I$ and $J$, respectively.
\medskip

As a first result, we show that the solution of 
the multi-penalty problem~\eqref{eq:multi} simultaneously
solves a related single-penalty problem.

\begin{lemma}\label{le:single}
  The pair $(u_{\alpha,\beta},v_{\alpha,\beta})$ solves~(\ref{eq:multi})
  if and only if
  \[
  v_{\alpha,\beta} = (\beta + A^*A)^{-1}(A^* y - A^*A u_{\alpha,\beta})
  \]
  and $u_{\alpha,\beta}$ solves the optimization problem
  \begin{equation}\label{eq:minu}
  \frac{1}{2}\|B_\beta u - y_\beta\|_2^2 + \alpha\|u\|_1 \to \min
  \end{equation}
  with
  \[
  B_\beta = \Bigl(\Id + \frac{AA^*}{\beta}\Bigr)^{-1/2}A
  \]
  and
  \[
  y_\beta = \Bigl(\Id+\frac{AA^*}{\beta}\Bigr)^{-1/2}y.
  \]
\end{lemma}

\begin{proof}
  We can solve the optimization problem in~\eqref{eq:multi} in two steps, first
  with respect to $v$ and then with respect to $u$.
  Assuming that $u$ is fixed, the optimality condition for $v$ in~\eqref{eq:multi} reads
  \begin{equation}\label{eq:vu}
  A^*(A(u+v)-y) + \beta v = 0.
  \end{equation}
  That is, for fixed $u$, the optimum in~\eqref{eq:multi} with respect to
  $v$ is obtained at
  \[
  v(u) := (\beta + A^*A)^{-1}(A^* y - A^*Au).
  \]
  Inserting this into the Tikhonov functional, we obtain the optimization
  problem
  \[
  \frac{1}{2}\|A(u+v(u))-y\|_2^2 + \alpha\|u\|_1 +
  \frac{\beta}{2}\|v(u)\|_2^2 \to \min_u.
  \]
  Using~\eqref{eq:vu}, we can write
  \[
  \begin{aligned}
  \frac{1}{2}\|A(u+v(u))-y\|_2^2
  &= \frac{1}{2}\langle A(u+v(u))-y,Au-y\rangle\\
  &\qquad\qquad{}+ \frac{1}{2}\langle A^*(A(u+v(u))-y),v(u)\rangle\\
  &= \frac{1}{2}\langle A(u+v(u))-y,Au-y\rangle
  - \frac{\beta}{2}\|v(u)\|_2^2.
  \end{aligned}
  \]
  Thus the optimization problem for $u$ simplifies to
  \begin{equation}\label{eq:optu1}
  \frac{1}{2}\langle A(u+v(u))-y,Au-y\rangle
  + \alpha\|u\|_1 \to \min_u.
  \end{equation}
  Now note that
  \[
  \begin{aligned}
    A(u+v(u))-y
    &= A(\Id-(\beta+A^*A)^{-1}A^*A)u - (\Id-A(\beta+A^*A)^{-1}A^*)y\\
    &= A(\Id+A^*A/\beta)^{-1}u - (\Id+AA^*/\beta)^{-1}y\\
    &= (\Id+AA^*/\beta)^{-1}(Au-y).
  \end{aligned}
  \]
  Inserting this equality in~\eqref{eq:optu1}, we obtain the
  optimization problem
  \[
  \frac{1}{2}\bigl\langle
  (\Id+AA^*/\beta)^{-1}(Au-y),Au-y\bigr\rangle + \alpha\|u\|_1
  \to \min_u,
  \]
  which is the same as~\eqref{eq:minu}.
\end{proof}

\begin{remark}
  As a consequence of Lemma~\ref{le:single}, we can apply
  the theory of $\ell^1$-regularization also to the multi-penalty
  setting we consider here. In particular, this yields, for fixed 
  $\beta > 0$, estimates of the form
  \[
  \|u^\dagger-u_{\alpha,\beta}\|_1 \le C_{1,\beta} \alpha + 
  C_{2,\beta} \frac{\|y_\beta-B_\beta u^\dagger\|_2^2}{\alpha}
  \]
  provided that $u^\dagger$ satisfies a source condition of the form 
  $B_\beta^*\eta \in \partial(\|u^\dagger\|_1)$ 
  with $|(B_\beta^*\eta)_i| < 1$ for every $i\not\in\mathop{{\rm supp}}(u^\dagger)$, 
  and the restriction of the mapping $B_\beta$ to the span of the
  support of $u^\dagger$ is injective
  (see~\cite{CPA:CPA20350}).
  Additionally, it is easy to show that these conditions hold
  for $B_\beta$ provided that they hold for $A$ and $\beta$ is
  sufficiently large.
  
\end{remark}

\section{Sets of Exact Support Recovery---Single-penalty Setting}

The main focus of this paper is the question whether multi-penalty
regularization allows for the exact recovery of the support
of the true solution $u^\dagger$ and how it compares to single-penalty
regularization.
Because, as we have seen in Lemma~\ref{le:single}, multi-penalty
regularization can be rewritten as single-parameter regularization
for the regularized operator $B_\beta$ and right hand side $y_\beta$,
we will first discuss recovery conditions in the single-parameter setting.

In order to find conditions for exact support recovery,
we first recall the necessary and sufficient optimality
condition for $\ell^1$-regularization:

\begin{lemma}\label{le:opt}
  The vector $u_\alpha$ minimizes
  \[
  T_\alpha(u) := \frac{1}{2}\|Au-y\|_2^2 + \alpha\|u\|_1,
  \]
  if and only if
  \[
  A^*(Au_\alpha-y) \in -\alpha\mathop{{\rm sgn}}(u_\alpha).
  \]
\end{lemma}

Using this result, we obtain a condition that guarantees
exact support recovery for the single-penalty method:

\begin{lemma}\label{le:support}
  We have $\mathop{{\rm supp}}(u_\alpha) = I$,
  if and only if there exists $w_\alpha \in (\R\setminus\{0\})^I$ such that
  \[
  \begin{aligned}
    A_I^*(A_Iw_\alpha-y) &= -\alpha \mathop{{\rm sgn}}(w_\alpha),\\
     \|A_J^*(A_Iw_\alpha-y)\|_\infty &\le \alpha,
  \end{aligned}
  \]
\end{lemma}

\begin{proof}
  This immediately follows from Lemma~\ref{le:opt} by
  testing the optimality conditions on the vector $u_\alpha$
  given by $(u_\alpha)_i = (w_\alpha)_i$ for $i \in I$
  and $(u_\alpha)_i = 0$ else.
\end{proof}
\medskip

Our main result concerning support recovery for single-parameter
regularization is the following:

\begin{proposition}\label{pr:cond_single}
  Assume that $A_I$ is injective and that
  \begin{equation}\label{eq:single_inftyrestriction}
  \|A_J^*A_I(A_I^*A_I)^{-1}\|_\infty < 1.
  \end{equation}
  Then the set $\mathcal{S}_{c,d,I}$ defined in~\eqref{eq:ScdI}
  is a set of exact support recovery for the
  unmixing problem whenever
  \begin{equation}\label{eq:cd}
  \frac{c}{d} >
  \frac{\|A_J^*(A_I(A_I^*A_I)^{-1}A_I^*-\Id)A\|_\infty\|(A_I^*A_I)^{-1}\|_\infty}{1-\|A_J^*A_I(A_I^*A_I)^{-1}\|_\infty}
  + \|(A_I^*A_I)^{-1}A_I^*A\|_{\infty}.
  \end{equation}
  Moreover, every parameter $\alpha > 0$ satisfying
  \begin{equation}\label{eq:single_alpha}
  \frac{d\|A_J^*(A_I(A_I^*A_I)^{-1}A_I^*-\Id)A\|_\infty}{1-\|A_J^*A_I(A_I^*A_I)^{-1}\|_\infty}
  \le \alpha
  < \frac{c-d\|(A_I^*A_I)^{-1}A_I^*A\|_{\infty}}{\|(A_I^*A_I)^{-1}\|_\infty}
  \end{equation}
  is admissible on $S_{c,d,I}$.
\end{proposition}

\begin{proof}
  First we note that the injectivity of $A_I$ implies that
  the mapping $A_I^*A_I$ is invertible. Thus the
  condition~\eqref{eq:single_inftyrestriction} actually makes sense.
  Moreover, the inequality~\eqref{eq:cd} is necessary
  and sufficient for the existence of $\alpha$ satisfying~\eqref{eq:single_alpha}.

  Now let $(u^\dagger,v) \in \mathcal{S}_{c,d,I}$ and assume that
  $\alpha$ satisfies~\eqref{eq:single_alpha}.
  We denote
  \[
  s_i^\dagger := \mathop{{\rm sgn}}(u_i^\dagger),
  \qquad
  i \in I,
  \]
  and define
  \[
  w_\alpha := u_I^\dagger + (A_I^*A_I)^{-1}A_I^*Av - \alpha (A_I^*A_I)^{-1}s^\dagger.
  \]
  Because $\|s^\dagger\|_{\infty} = 1$, it follows from the second inequality
  in~\eqref{eq:single_alpha} that
  \[
  |(w_\alpha)_i - u_i^\dagger|
  \le \|(A_I^*A_I)^{-1}A_I^*A\|_\infty\|v\|_\infty + \alpha \|(A_I^*A_I)^{-1}\|_\infty
  < c \le |u_i^\dagger|,
  \]
  and therefore
  \[
  \mathop{{\rm sgn}}(w_\alpha) = \mathop{{\rm sgn}}(u_I^\dagger) = s^\dagger.
  \]
  Thus $w_\alpha$ actually satisfies the equation
  \[
  \begin{aligned}
  w_\alpha &= u_I^\dagger + (A_I^*A_I)^{-1}A_I^*Av - \alpha (A_I^*A_I)^{-1}\mathop{{\rm sgn}}(w_\alpha)\\
  &= (A_I^*A_I)^{-1}A_I^*(A_Iu_I^\dagger + Av) - \alpha (A_I^*A_I)^{-1}\mathop{{\rm sgn}}(w_\alpha)\\
  &= (A_I^*A_I)^{-1}A_I^*y- \alpha (A_I^*A_I)^{-1}\mathop{{\rm sgn}}(w_\alpha),
  \end{aligned}
  \]
  and thus
  \[
  A_I^*(A_Iw_\alpha-y) = -\alpha\mathop{{\rm sgn}}(w_\alpha),
  \]
  which is the first condition in Lemma~\ref{le:support}.

  It remains to show that
  \[
  \|A_J^*(A_Iw_\alpha-y)\|_\infty\le \alpha.
  \]
  However,
  \[
  \begin{aligned}
    A_J^*(A_Iw_\alpha-y)
    &= A_J^*(A_Iw_\alpha-A_Iu_I^\dagger-Av)\\
    &= A_J^*(A_I(A_I^*A_I)^{-1}A_I^*Av-Av-\alpha A_I(A_I^*A_I)^{-1}s^\dagger),
  \end{aligned}
  \]
  and thus
  \[
  \begin{aligned}
  \|A_J^*(A_Iw_\alpha-y)\|_\infty
  &\le d\|A_J^*(A_I(A_I^*A_I)^{-1}A_I^*-\Id)A\|_\infty\\
  &\qquad\qquad\qquad\qquad+ \alpha\|A_J^*A_I(A_I^*A_I)^{-1}\|_\infty.
  \end{aligned}
  \]
  Now the first inequality in~\eqref{eq:single_alpha} implies
  that this term is smaller than $\alpha$.
  Thus $w_\alpha$ satisfies the conditions of Lemma~\ref{le:support},
  and thus $\mathop{{\rm supp}}(u_\alpha) = I$.
\end{proof}

\begin{remark}
  In the case where $A = \Id$ is the identity operator, the conditions
  above reduce to the conditions that $c > 2d$ and $d \le \alpha < c-d$.
  Since $\ell^1$-regularization in this setting reduces to soft 
  thresholding, these conditions are very natural and are
  actually both sufficient and necessary:
  Since the noise may componentwise reach the value of $d$,
  it is necessary to choose a regularization parameter of at least
  $d$ in order to remove it. However, on the support $I$
  of the signal, the smallest values of the noisy signal value
  are at least of size $c-d$. Thus they are retained as long
  as the regularization parameter does not exceed this value.

  For more complicated operators $A$, the situation is similar, i.e., a too small regularization parameter $\alpha$ is not
  able to remove all the noise, while a too large one destroys
  part of the signal as well.
  The exact bounds for the admissible regularization parameters,
  however, are much more complicated.
\end{remark}

\section{Sets of Exact Support Recovery---Multi-Penalty Setting}

We now consider the setting of multi-penalty regularization for the
solution of the unmixing problem. Applying Lemma~\ref{le:single}, we
can treat multi-penalty regularization with the same methods as
single-penalty regularization.
To that end, we introduce the regularized operator
\begin{equation}\label{eq:Abeta}
A_\beta := \Bigl( I + \frac{AA^*}{\beta}\Bigr)^{-1}A.
\end{equation}
In particular, we have with the notation of Lemma~\ref{le:support}
that
$B_\beta^* B = A_\beta^*A$ and $B_\beta^*y_\beta = A_\beta^*y$.

As a first result, we obtain the following
analogon to Lemma~\ref{le:support}:

\begin{lemma}\label{le:support_multi}
  We have $\mathop{{\rm supp}}(u_{\alpha,\beta}) = I$,
  if and only if there exists $w_\alpha \in (\R\setminus\{0\})^I$ such that
  \[
  \begin{aligned}
    A_{\beta,I}^*(A_Iw_{\alpha,\beta}-y) &= -\alpha \mathop{{\rm sgn}}(w_{\alpha,\beta}),\\
    \|A_{\beta,J}^*(A_Iw_\alpha-y)\|_\infty &\le \alpha,
  \end{aligned}
  \]
\end{lemma}

\begin{proof}
  Applying Lemma~\ref{le:opt} to the single-penalty
  problem~\eqref{eq:minu}, we obtain the conditions
  \[
  \begin{aligned}
    B_{\beta,I}^*(B_{\beta,I}w_{\alpha,\beta}-y_\beta) &= -\alpha \mathop{{\rm sgn}}(w_{\alpha,\beta}),\\
    \|B_{\beta,J}^*(B_{\beta,I}w_{\alpha,\beta}-y_\beta)\|_\infty &\le \alpha.
  \end{aligned}
  \]
  Now the claim follows from the equalities
  \[
  \begin{aligned}
  B_{\beta,I}^*B_{\beta,I} &= A_{\beta,I}^*A_I,&
  B_{\beta,I}^*y_\beta &= A_{\beta,I}^*y,\\
  B_{\beta,I}^*B_{\beta,I} &= A_{\beta,I}^*A_I,&
  B_{\beta,I}^*y_\beta &= A_{\beta,I}^*y.
  \end{aligned}
  \]
\end{proof}

Since the proof of Proposition~\ref{pr:cond_single} only
depends on the optimality conditions and the representation
of the data as $y = Au^\dagger + Av$, we immediately
obtain a generalization of Proposition~\ref{pr:cond_single}
to the multi-penalty setting.

\begin{proposition}\label{pr:cond_multi}
  Assume that $0 < \beta < \infty$ is such that
  \begin{equation}\label{eq:cond_multi_nec}
  \|A_{\beta,J}^*A_{I}(A_{\beta,I}^*A_{I})^{-1}\|_\infty < 1.
  \end{equation}
  Then the set $\mathcal{S}_{c,d,I}$ is a set of exact support recovery for the
  unmixing problem in the multi-penalty setting whenever
  \begin{multline}
    \frac{c}{d} >\|(A_{\beta,I}^*A_I)^{-1}A_{\beta,I}^*A\|_{\infty}\\
    + \frac{\|A_{\beta,J}^*(A_I(A_{\beta,I}^*A_I)^{-1}A_{\beta,I}^*-\Id)A\|_\infty
      \|(A_{\beta,I}^*A_I)^{-1}\|_\infty}{1-\|A_{\beta,J}^*A_I(A_{\beta,I}^*A_I)^{-1}\|_\infty}.
    \label{eq:cond_multi_1}
  \end{multline}
  Moreover, all the pairs of parameter $(\alpha,\beta)$
  satisfying~\eqref{eq:cond_multi_1} and
  \begin{equation}\label{eq:multi_alpha}
  \frac{d\|A_{\beta,J}^*(A_I(A_{\beta,I}^*A_I)^{-1}A_{\beta,I}^*-\Id)A\|_\infty}{1-\|A_{\beta,J}^*A_I(A_{\beta,I}^*A_I)^{-1}\|_\infty}
  \le \alpha
  < \frac{c-d\|(A_{\beta,I}^*A_I)^{-1}A_{\beta,I}^*A\|_{\infty}}{\|(A_{\beta,I}^*A_I)^{-1}\|_\infty}
  \end{equation}
  are admissible on $\mathcal{S}_{c,d,I}$.
\end{proposition}

\begin{proof}
  The proof is analogous to the proof of Proposition~\ref{pr:cond_single}.
\end{proof}

\begin{remark}
  We note that the condition
  \[
  \|A_J^*A_I(A_I^*A_I)^{-1}\|_\infty < 1
  \]
  implies the analogous inequality for $A_\beta$ provided that $\beta$
  is sufficiently large. Similarly, if $\alpha$ satisfies the
  conditions in Proposition~\ref{pr:cond_single} that guarantee
  admissibility on $S_{c,d,I}$, the pair $(\alpha,\beta)$ will satisfy
  the conditions for admissibility in Proposition~\ref{pr:cond_multi}
  provided that $\beta$ is sufficiently large. The converse, however,
  need not be true: If the pair $(\alpha,\beta)$ is admissible for exact
  support recovery on $\mathcal{S}_{c,d,k}$ with multi-penalty
  regularization, it need not be true that the single parameter
  $\alpha$ is admissble for the single-penalty setting as well.
  Examples where this actually happens can be found in Section~\ref{se:valid}
  (see in particular Table~\ref{tb:cond}).
\end{remark}

\section{Admissible parameters}

As a consequence of Propositions~\ref{pr:cond_single} and~\ref{pr:cond_multi},
we obtain that the condition
\begin{equation}\label{eq:cond_k}
\sup_{|I| \le k} \|A_{\beta,J}^*A_{I}(A_{\beta,I}^*A_{I})^{-1}\|_\infty < 1
\end{equation}
is sufficient for $S_{c,d,k}$ to be a set of exact support recovery
for the unmixing problem, provided that the ratio $c/d$ is sufficiently large;
the condition for the single-parameter case
can be extracted from~\eqref{eq:cond_k} by setting $\beta = \infty$,
in which case $A_\beta$ reduces to $A$.

Now define the signal-to-noise ratio of a pair $(u,v)$ as
\[
\rho(u,v) := \frac{\inf\bigl\{|u_i| : i \in \mathop{{\rm supp}}(u)\bigr\}}{\|v\|_\infty}.
\]
That is, $\rho(u,v)$ is the ratio of the smallest significant value
of the signal $u$, and the largest value of the noise $v$.
Denote moreover
\begin{multline*}
R(\beta,k) := \max_{|I|\le k} \Biggl\{\frac{\|A_{\beta,J}^*(A_I(A_{\beta,I}^*A_I)^{-1}A_{\beta,I}^*-\Id)A\|_\infty
      \|(A_{\beta,I}^*A_I)^{-1}\|_\infty}{1-\|A_{\beta,J}^*A_I(A_{\beta,I}^*A_I)^{-1}\|_\infty}\\
   + \|(A_{\beta,I}^*A_I)^{-1}A_{\beta,I}^*A\|_{\infty}\Biggr\}.
\end{multline*}
Then the inequality~\eqref{eq:cond_multi_1} implies that
multi-penalty regularization with parameter $\beta$
allows for the recovery of the support of
$k$-sparse vectors $u$ from data $A(u+v)$ provided the signal-to-noise ratio
of the pair $(u,v)$ satisfies
\[
\rho(u,v) > R_{\beta,k}.
\]
Whenever the signal-to-noise ratio is larger than $R_{\beta,k}$,
we can recover the support of the vector $u$ with \emph{some}
regularization parameter $\alpha$.
There are, however, upper and lower limits for the admissible
parameters $\alpha$, given by inequality~\eqref{eq:multi_alpha}.
In order to visualize them, we consider instead the ratio
\[
\theta(\alpha,v) := \frac{\alpha}{\|v\|_\infty}.
\]
Defining
\[
\Theta^{\min}_{\beta,k} := \max_{|I|\le k}
\frac{\|A_{\beta,J}^*(A_I(A_{\beta,I}^*A_I)^{-1}A_{\beta,I}^*-\Id)A\|_\infty}{1-\|A_{\beta,J}^*A_I(A_{\beta,I}^*A_I)^{-1}\|_\infty}
\]
and
\[
\Theta^{\max}_{\beta,k}(\vartheta)
:= \min_{|I|\le k}\frac{\vartheta-\|(A_{\beta,I}^*A_I)^{-1}A_{\beta,I}^*A\|_{\infty}}{\|(A_{\beta,I}^*A_I)^{-1}\|_\infty},
\]
we then obtain the condition
\begin{equation}\label{eq:theta}
\Theta^{\min}_{\beta,k} \le \theta(\alpha,v) < \Theta^{\max}_{\beta,k}\bigl(\rho(u,v)\bigr)
\end{equation}
for exact support recovery.
If the ratio $\theta(\alpha,v)$ is smaller than $\Theta^{\min}_{\beta,k}$,
then it can happen that some of the noise $v$ is not filtered
out by the regularization method.
On the other hand, if $\theta(\alpha,v)$ is larger than
$\Theta^{\max}_{\beta,k}\bigl(\rho(u,v)\bigr)$, then some parts of
the signal $u^\dagger$ might actually be lost because of
the regularization.

We note that the function $\Theta^{\max}_{\beta,k}$ is piecewise linear
and concave, and $\lim_{\vartheta\to \infty}\Theta^{\max}_{\beta,k}(\vartheta) = +\infty$.
Thus the region of admissible parameters defined by~\eqref{eq:theta}
is a convex and unbounded polyhedron.
Moreover, we have that $\Theta_{\beta,k}^{\max}(R_{\beta,k}) = \Theta_{\beta,k}^{\min}$.
Additionally, we note that the behaviour of the function $\Theta_{\beta,k}^{\max}$
near infinity is determined by the term
\[
\Sigma_{\beta,k} := \max_{|I| \le k} \|(A_{\beta,I}^*A_I)^{-1}\|_\infty.
\]
If this value is small, then the slope of the function
$\Theta_{\beta,k}^{\max}(\vartheta)$ for large values of $\vartheta$
is large, and thus the set of admissible parameter grows
fast with increasing signal-to-noise ratio.
If, on the other hand, $\Sigma_{\beta,k}$ is large, then the
set of admissible parameters is relatively small
even for large signal-to-noise ratio.
Thus $\Sigma_{\beta,k}$ can be reasonably interpreted as
the sensitivity of multi-parameter regularization with
respect to parameter choice.
The larger $\Sigma_{\beta,k}$ is, the more precise the parameter $\alpha$
has to be chosen in order to guarantee exact support recovery.

\section{Numerical Validation}\label{se:valid}

The main motive behind the study and application of multi-penalty
regularization is the problem that $\ell^1$-regularization is often
not capable to identify the support of signal correctly (see \cite{arfopeXX} and references therein).
Including the additional $\ell^2$-regularization term, however, might
lead to an improved performance in terms of support recovery, because
we can expect that the $\ell^2$-term takes care of all the small noise
components. 

In order to verify  this observation, a series of numerical experiments was performed, in which we illustrate for which parameters and Gaussian matrices the conditions for support recovery derived in the previous
section were satisfied. In addition, we studied whether the
inclusion of the $\ell^2$-term indeed increases the performance.

In a first set of experiments, we have generated a set of 20 Gaussian
random matrices of different sizes and have tested for each
three-dimensional subspace spanned by the basis elements whether the
condition~\eqref{eq:cond_multi_nec} is satisfied, first for
the single-penalty case, and then for the multi-penalty case with
different values of $\beta$.
The results for matrices of dimensions 30 times 60 and 40 times 80,
respectively, are summarized in Table~\ref{tb:cond} and
Figure~\ref{fi:cond}.

As to be expected from the bad numerical performance of
$\ell^1$-regularization in terms of support recovery, the
inequality~\eqref{eq:single_inftyrestriction} fails in a relatively
large number of cases, especially when the discrepancy between the
dimension $N$ of the vectors to be recovered and the number of
measurements $m$ is quite large. For instance, in the case $N=60$ and
$m=30$, the condition most of the time failed for more than half of
the three-dimensional subspaces. In contrast, the corresponding
condition~\eqref{eq:cond_multi_nec} for multi-parameter regularization
fails in the same situation only for about an eighth of the subspaces
if $\beta = 1$, and in even fewer cases for $\beta = 0.1$.

For other combinations of dimensionality of the problem and number of
measurements, the situation is similar. Introducing the additional
$\ell^2$-penalty term always allows for the exact support
reconstruction on a larger number of subspaces than single-penalty
regularization. Additionally, the results indicate that the number of
recoverable subspaces increases with decreasing $\beta$.

\begin{table}[h]
\begin{tabular}{r||c|ccc}
$m=30$  &\ Single-penalty\ {} & \multicolumn{3}{c}{Multi-penalty}\\
$N=60$ & &\quad $\beta = 10$ \quad{} &\quad $\beta = 1$\quad{} &\quad $\beta = 0.1$\quad{} \\
  \hline
Median & 0.5425 & 0.3814 & 0.1214 & 0.0623\\
Mean & 0.5559 & 0.3922 & 0.1225 & 0.0635\\
Std.~deviation & 0.05652 & 0.04142 & 0.01518 & 0.01083\\
\end{tabular}
\medskip

\begin{tabular}{r||c|ccc}
$m=40$  &\ Single-penalty\ {} & \multicolumn{3}{c}{Multi-penalty}\\
$N=80$  & &\quad $\beta = 10$ \quad{} &\quad $\beta = 1$\quad{} &\quad $\beta = 0.1$\quad{} \\
  \hline
Median & 0.2696 & 0.1523 & 0.0396 & 0.0256\\
Mean &   0.2746 & 0.1547 & 0.0413 & 0.0262\\
Std.~deviation & 0.03060 & 0.01848 & 0.00659 & 0.00447 \\
\end{tabular}
\caption{\label{tb:cond}
    Percentage of 3-sparse subspaces for which the condition~\eqref{eq:cond_multi_nec}
  failed. The condition was tested on samples of 20 Gaussian random matrices
  of dimensions 30 times 60 (upper table) and 40 times 80 (lower table).
  Other combinations of dimensionality and number of measurements showed
  qualitatively similar results.
}
\end{table}

\begin{figure}[h]
  \includegraphics[width=0.45\textwidth]{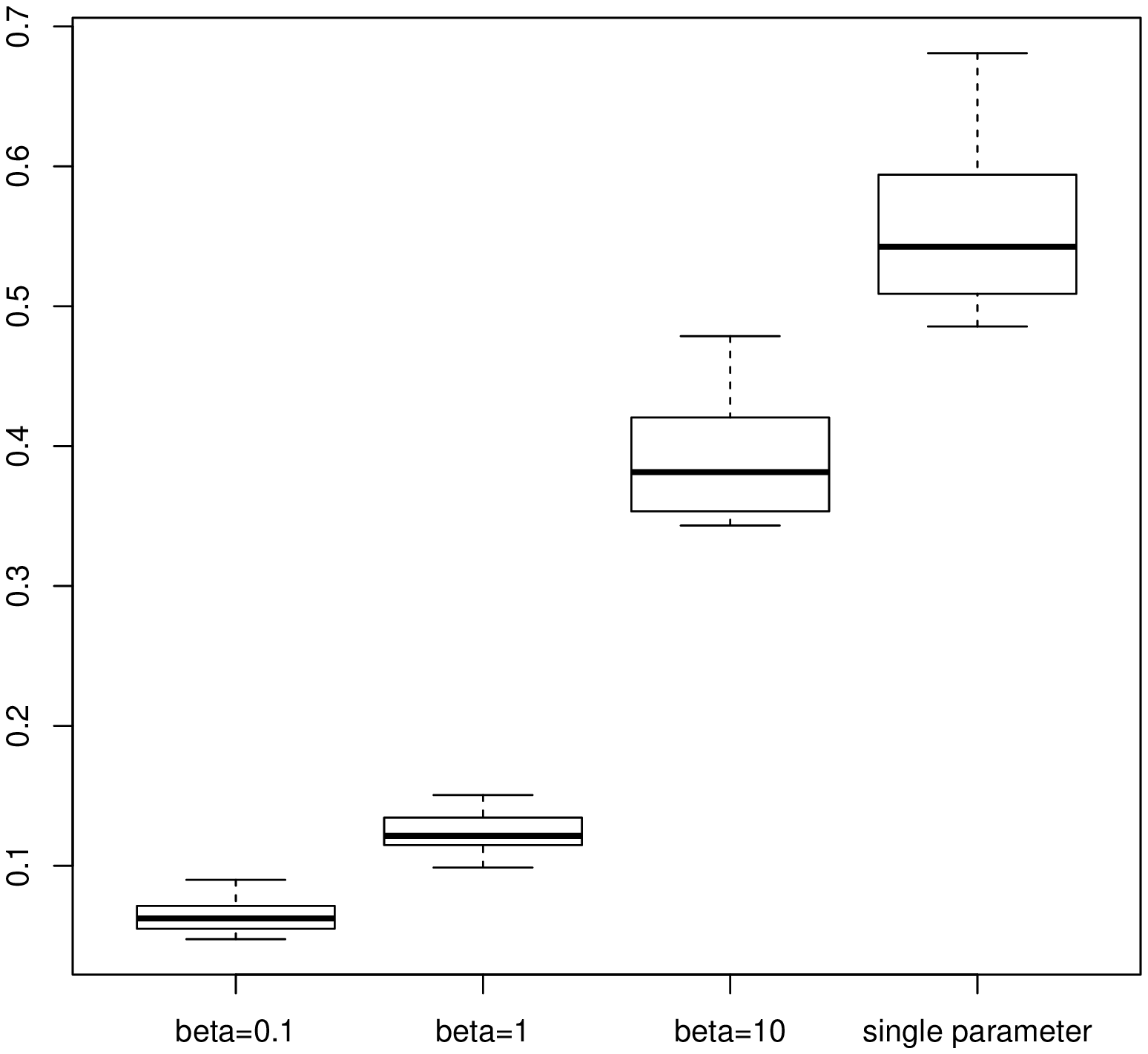}
  \includegraphics[width=0.45\textwidth]{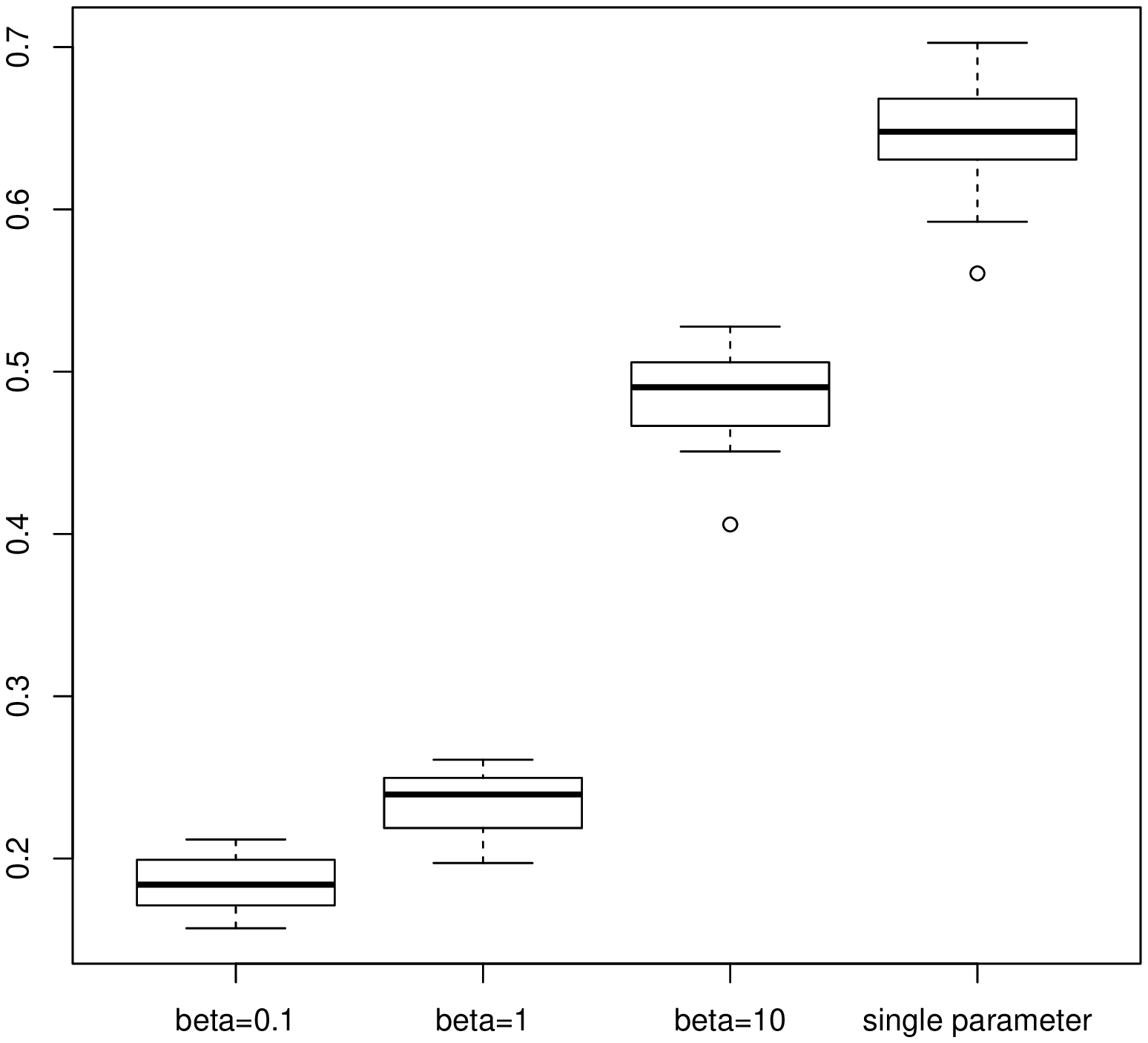}\\
  \includegraphics[width=0.45\textwidth]{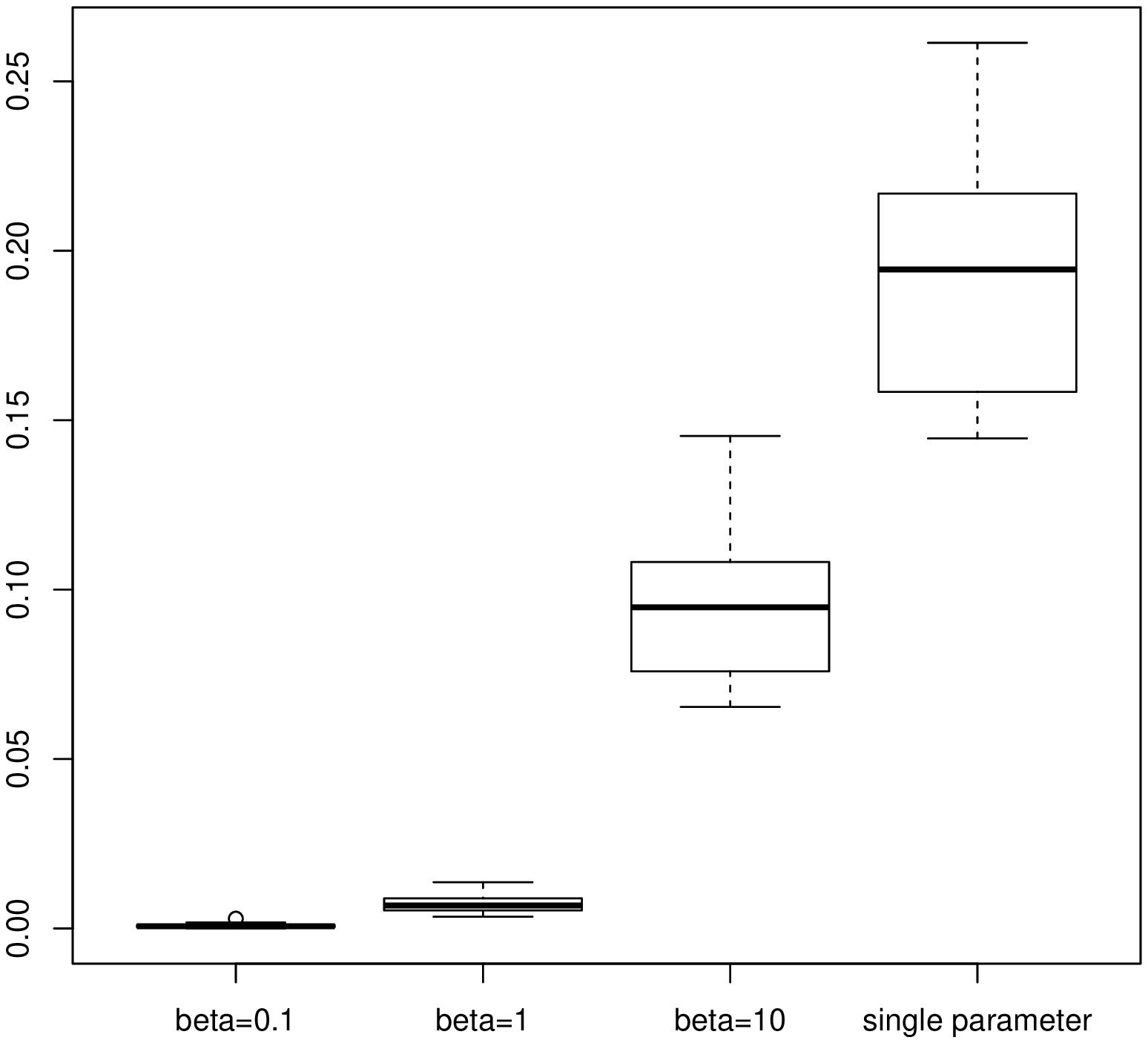}
  \includegraphics[width=0.45\textwidth]{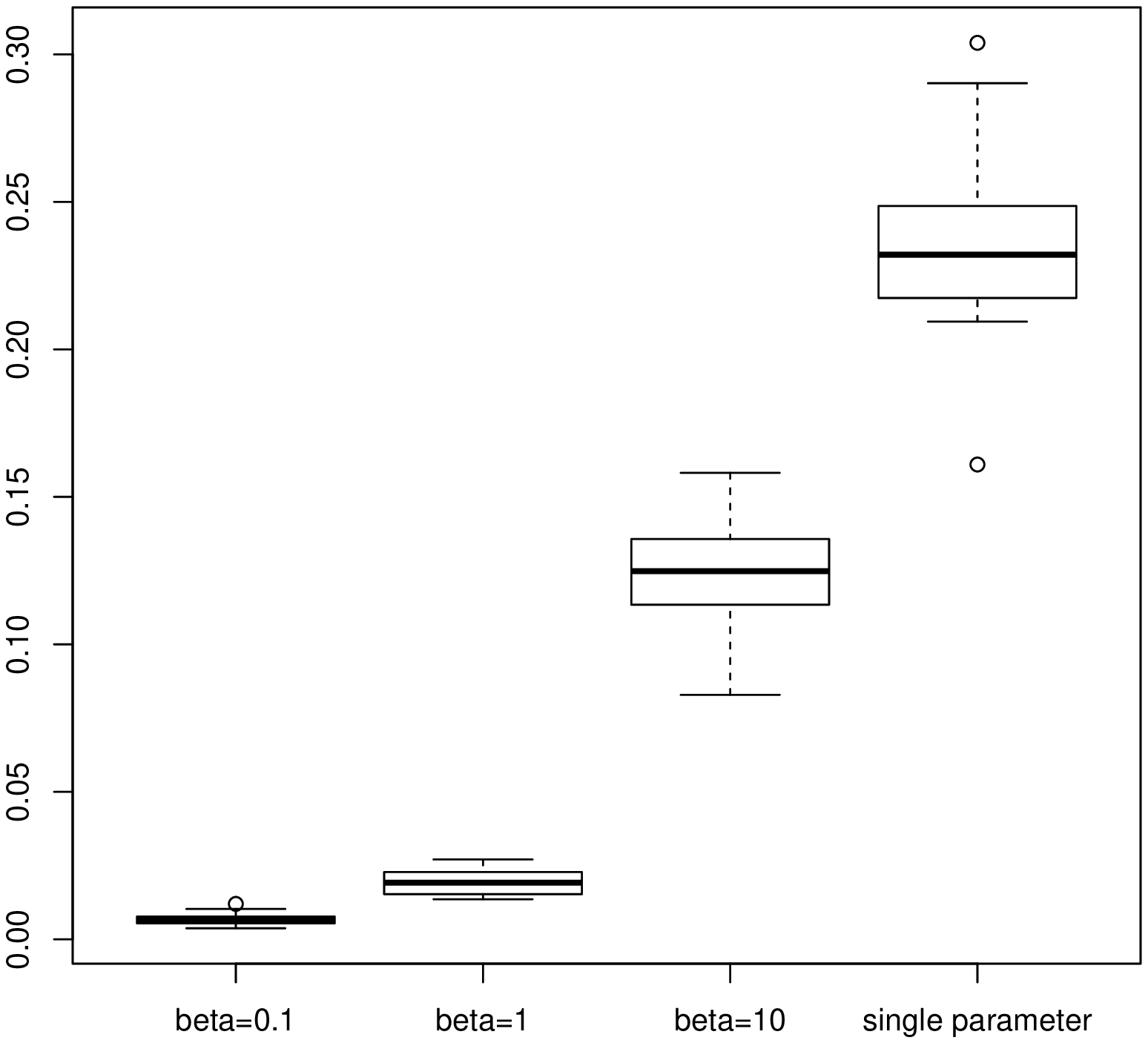}
  \caption{\label{fi:cond}
    Percentage of 3-sparse subspaces for which the condition~\eqref{eq:cond_multi_nec}
    failed. For each of the different settings the condition
    has been tested on 20 Gaussian random matrices
    for multi-parameter regularization with $\beta = 0.1$,
    $\beta = 1$ and $\beta = 10$, and for single parameter
    regularization.
    \emph{Upper left:} Dimension 30 times 60.
    \emph{Upper right:} Dimension 30 times 80.
    \emph{Lower left:} Dimension 40 times 60.
    \emph{Lower right:} Dimension 40 times 80.} 
\end{figure}

In the case where $N = 80$ and $m=60$ (that is, we want to reconstruct
$80$-dimensional vectors from 60 measurements), the sufficient
condition~\eqref{eq:cond_multi_nec} for multi-penalty regularization
was satisfied in our numerical experiments for all $3$-sparse
subspaces for parameters $\beta$ smaller than 5.
In this situation, we have therefore additionally computed the
significant values $R_{\beta,3}$ and $\Sigma_{\beta,3}$,
that is, the minimal recoverable signal-to-noise ratio and the
parameter sensitivity for three-dimensional subspaces.
The results of these numerical experiments are shown in
Table~\ref{tb:beta} and Figure~\ref{fi:beta}.

The results indicate that decreasing the regularization parameter
$\beta$ tends to decrease the necessary signal-to-noise ratio
$R_{\beta,k}$ as well. While a regularization parameter $\beta = 5$
required always an unreasonably large signal-to-noise ratio in order
to guarantee the recoverability of the support, the ratios for $\beta
= 0.5$ or $\beta = 0.1$ turned out to be much more
reasonable. However, the results also showed that decreasing the
regularization parameter need not necessarily have a beneficial effect
on the recoverability. For some matrices it happened that the
necessary signal-to-noise ratio $R_{\beta,k}$ increased while the
regularization parameter $\beta$ was decreased
(see Figure~\ref{fi:beta}, left).

Additionally, the results indicate that the parameter sensitivity
$\Sigma_{\beta,k}$ increases considerably as $\beta$ decreases (see
Figure~\ref{fi:beta}, right). As a consequence, the range of
admissible parameters $\alpha$ tends to be significantly smaller for
smaller $\beta$, and it can happen much more easily that the combined
effect of $\ell^2$ and $\ell^1$-regularization leads to a
classification of signals as noise. While multi-penalty regularization
with a small parameter $\beta$ might therefore lead to a better
necessary signal-to-noise ratio for recovery, it requires at the same
time a better balance between the two involved parameters $\alpha$ and
$\beta$.

\begin{figure}[h]
  \includegraphics[width=0.45\textwidth]{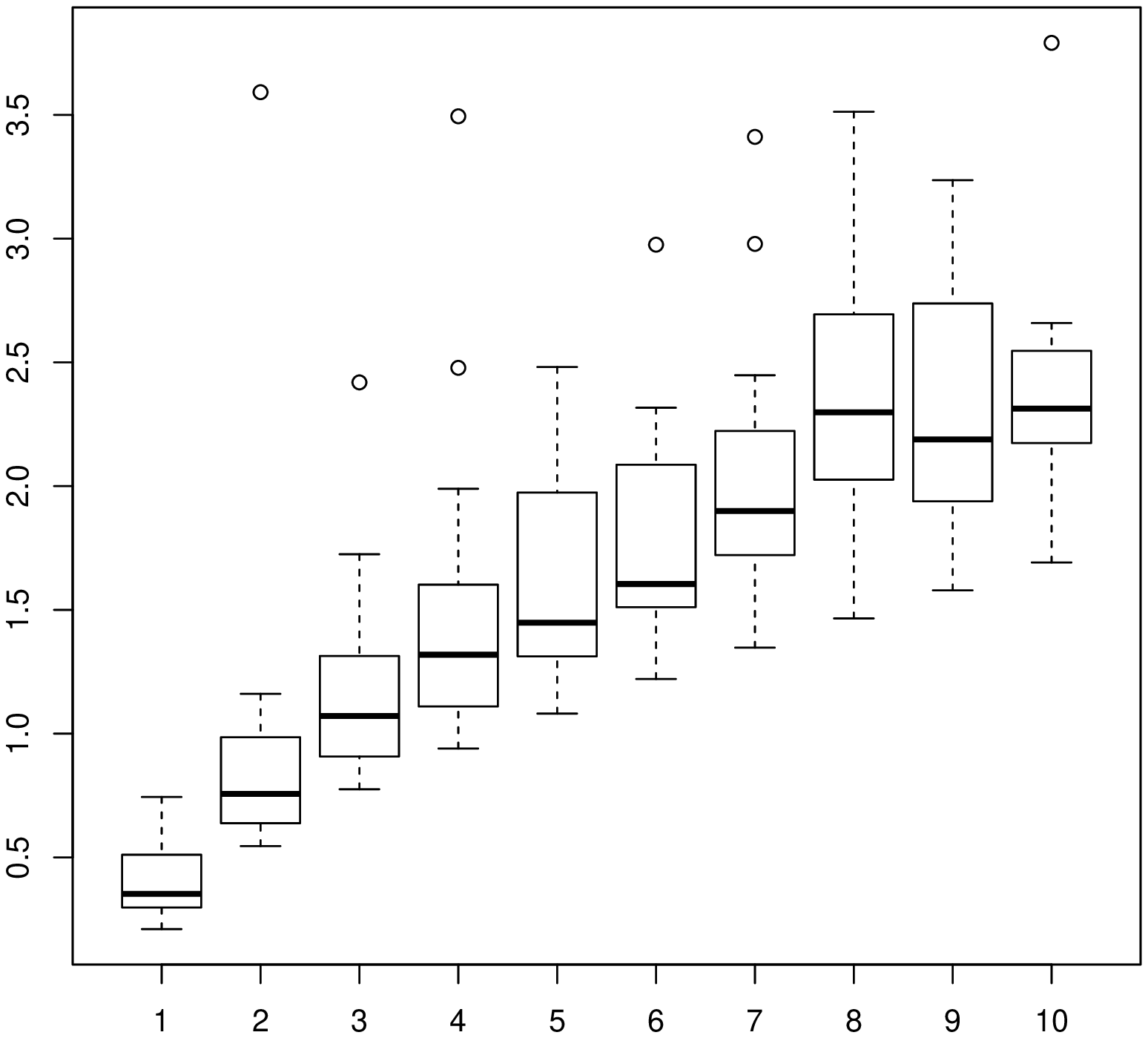}
  \includegraphics[width=0.45\textwidth]{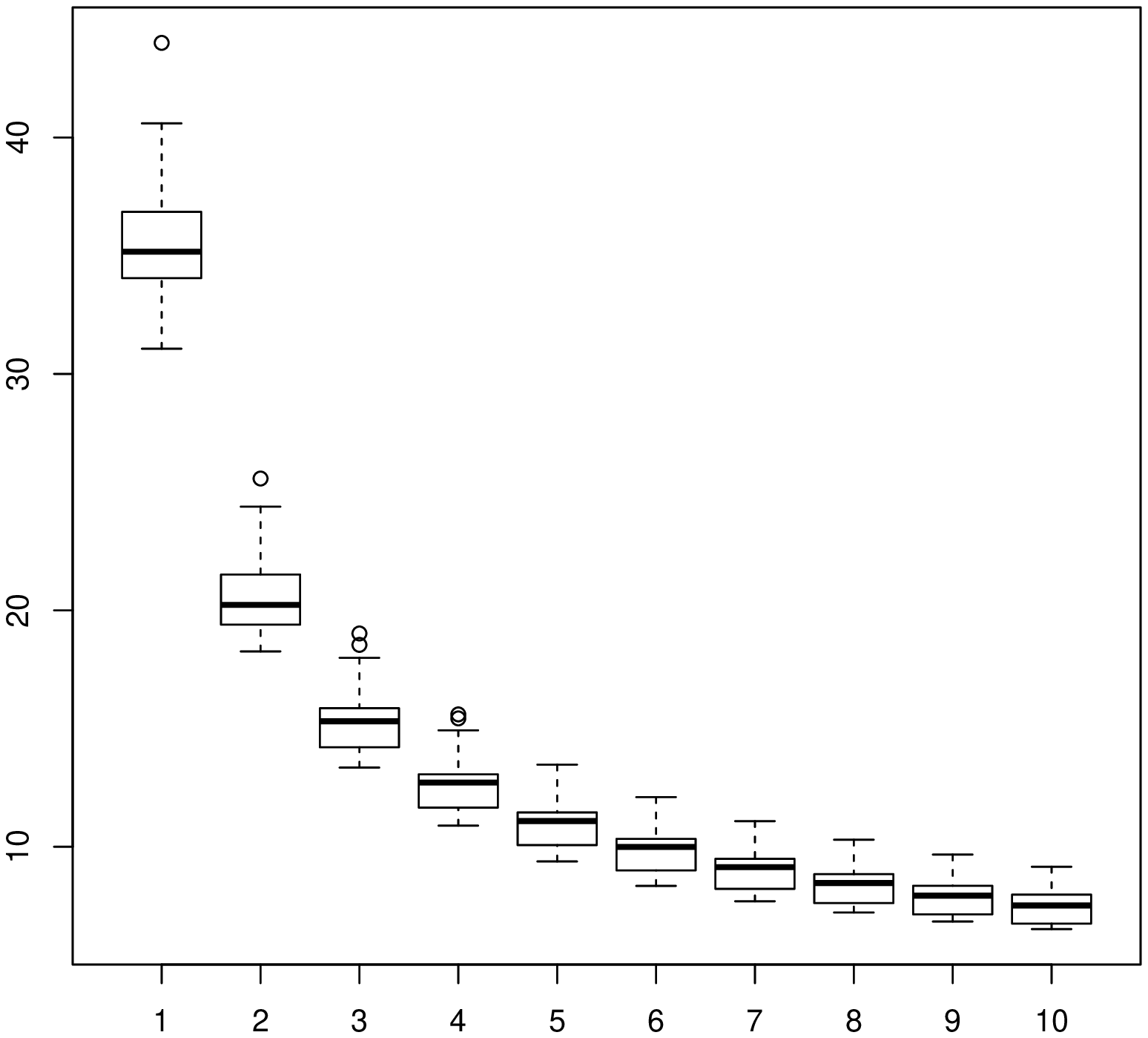}
  \caption{\label{fi:beta}
    Influence of the regularization parameter $\beta$
    on the set of exact support recovery and the
    admissible parameters.
    \emph{Left:} Logarithm of base 10 of the recoverable
    signal-to-noise ratio $R_{\beta,3}$ for different values
    of $\beta$.
    \emph{Right:} The sensitivity $\Sigma_{\beta,3}$ for different
    values of $\beta$.\newline
    The dimension of the matrix is in all cases 60 times 80,
    and 20 Gaussian random matrices have been
    used for each parameter $\beta$.} 
\end{figure}

\begin{table}[h]
  \begin{tabular}{rr|rrrrr}
    & &\multicolumn{5}{c}{$\beta$}\\
    & & $0.1$ & $0.3$ & $0.5$\hfil & $1$\hfil & $5$\hfil \\
    \hline
    & minimum & 1.623 & 5.961 & 12.05 & 49.14 & 3599.3 \\
    $R_{\beta,3}$ & median & 2.252 & 11.84 & 28.17 & 205.7 & 15041.4 \\
    & maximum & 5.546 & 262.5 & 302.85 & 6178.4 & 203621.6\\
    \hline
    & minimum & 31.07 & 13.35 & 9.379 & 6.517 & 3.793 \\
    $\Sigma_{\beta,3}$ & median & 35.17 & 15.31 & 11.086 & 7.517 & 4.278 \\
    & maximum & 44.00 & 19.02 & 13.474 & 9.154 & 4.977 \\
  \end{tabular}
  \caption{\label{tb:beta}
    Influence of the parameter $\beta$ on the values
    of $R_{\beta,k}$ and $\Sigma_{\beta,k}$. The values have been
    computed for 20 Gaussian random matrices of dimension
    60 times 80 with $k=3$.}
\end{table}

\section{Applications. Numerical experiments}

In order to support our theoretical findings even further, we present in this
section some statistical data obtained by solving series of
compressive sensing problems  by means of  multi-penalty  and
$\ell^1$-regularization. Similarly to \cite{naumpeter, arfopeXX} we
consider in our numerical experiments the model problem of the type 
\[
	y= T (u^\dagger + v),
\]
where $T \in  \R^{m \times N}$ is an i.i.d.\ Gaussian matrix, $u^\dag$ is a sparse vector and $v$ is a noise vector. The choice of $T$ corresponds to compressed sensing measurements \cite{FoRa13}.

In the experiments, we consider 30 problems of this type with $u^\dag$ randomly generated, $\inf_{i \in I} | u^\dag| > 1.5$  and $\#\mathop{{\rm supp}}(u^\dag) = 7,$ and $v$ is a random vector whose components are uniformly distributed on $[-1,1]$, and normalized such that $\|v \|_\infty =0.3.$ We also consider the Gaussian matrices of the size $m = 50$, $N = 100.$ 

In order to approximate minimizers of the multi-penalty (\ref{eq:multi}) and the corresponding single-penalty functional  we use the iterative soft-thresholding algorithm \cite{Fornasier07}.
The regularization parameters $\alpha$ and $\beta$ were chosen from the grid $Q_{\alpha_0}^k \times Q_{\beta_0}^k$, where $Q_{\alpha_0}^k:=\{\alpha= \alpha_i = \alpha_0 k^i\, ,\alpha_0 = 0.0002,k=1.25,i=0,\ldots,50\}$, and $Q_{\beta_0}^k:=\{\beta = \beta_i = \beta_0 k^i,\beta_0 = 0.01,q=1.15,i=0,\ldots,30\}$. For all possible combinations of  $(\alpha,\beta)$  we run the iterative soft-thresholding algorithm with fifty inner loop iterations and starting values $u^{(0)}=v^{(0)}=0$. 

In order to assess the obtained results, we compare the performance of
the considered regularization schemes. We measure the approximation
error (AE) by $\| u - u^\dag \|_{2}$, as well as by the number of
elements in the symmetric difference (SD) $\#(\mathop{{\rm supp}}(u)
\Delta \mathop{{\rm supp}}(u^\dag))$. The SD is defined as follows: $i
\in \mathop{{\rm supp}}(u) \Delta \mathop{{\rm supp}}(u^\dag)$ if and
only if  either $i \notin \mathop{{\rm supp}}(u)$ and
$i\in\mathop{{\rm supp}}(u^\dag)$ or $i \in\mathop{{\rm supp}}(u)$ and
$i \notin\mathop{{\rm supp}}(u^\dagger)$.

For each problem we compute the best multi-penalty solution $u^\dag =
u^\dag(\alpha,\beta)$ meaning that no  other pairs $(\alpha,\beta)\in
Q_{\alpha_0}^k\times Q_{\beta_0}^k$ can improve the accuracy of the
algorithm.  
Simultaneously, for each problem from our data set we compute the best
mono-penalty solution $u^\dag = u^\dag(\alpha)$. 
Then,  we compute the mean value of the AE and SD. The respective
results are shown in table~\ref{tb:mp_sp}.
Additionally, figure~\ref{fi:comp} shows an example of the typical
results obtained for single- and for multi-penalty regularization.

\begin{table}[h]
  \begin{tabular}{rr|rrrr}
    & & $AE$ & $SD$ & $\alpha$ & $\beta$ \hfil \\
    \hline
    & minimum & 5.00  & 3 & 4.59 &  \\
    $SP$ & mean &  11.21 & 6 &  5.74&  \\
    & maximum &  13.74 & 8 & 11.21 &  \\
    \hline
    & minimum & 1.05  & 1 & 3.67 & 0.76  \\
    $MP$ & mean & 5.92 & 3 & 5.74 & 7.09  \\
    & maximum & 8.55 & 5 & 8.97& 9.42  \\
  \end{tabular}
  \caption{\label{tb:mp_sp}
   For 30 problems for the solution of the single-penalty (upper panel) as well as multi-penalty regularization (lower panel) the minimum / maximum AE, SD and optimal values of the regularization parameters are provided. The mean values are also provided.
}
\end{table}

\begin{figure}[h]

   \includegraphics[width=0.95\textwidth]{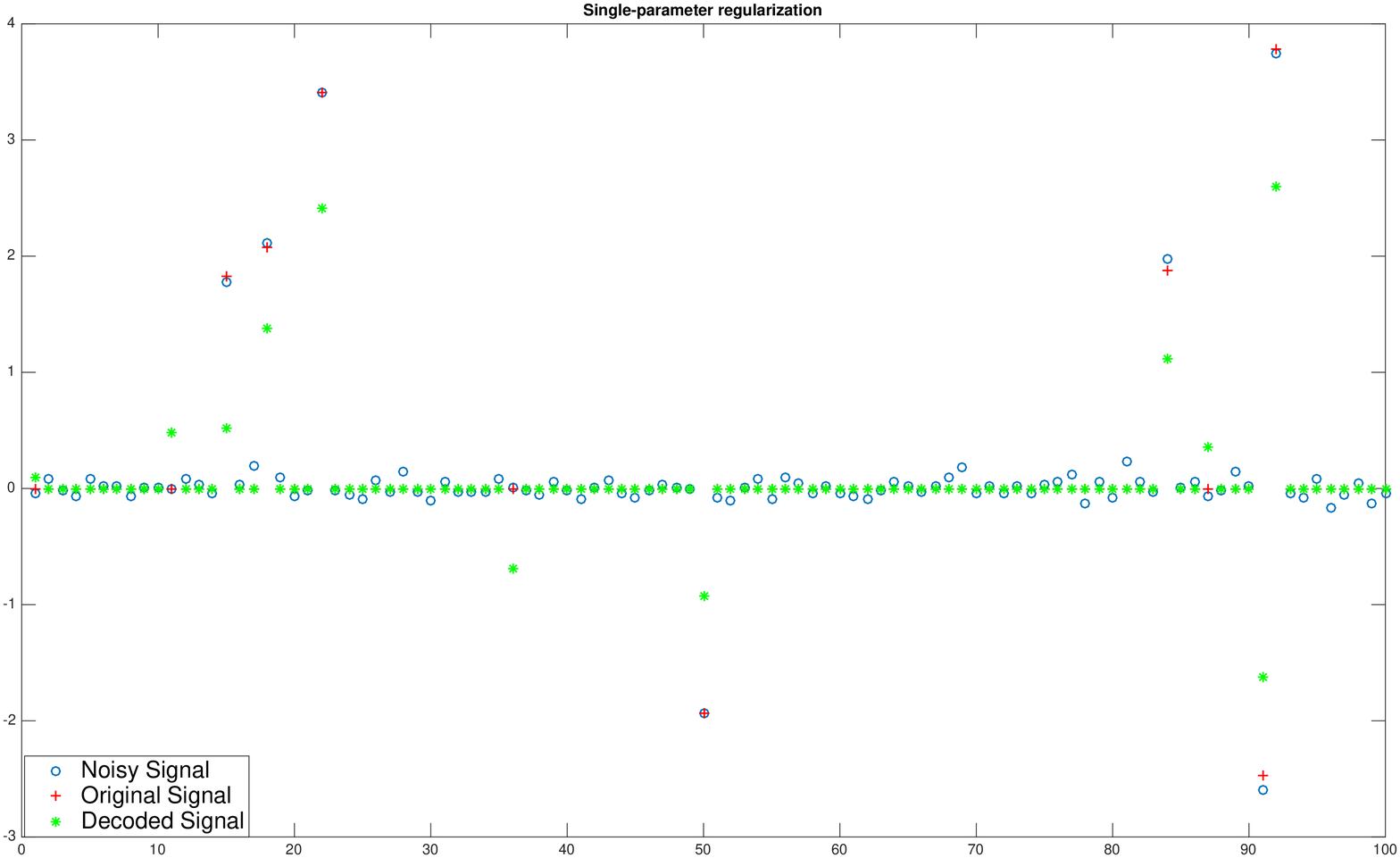}
    \includegraphics[width=0.95\textwidth]{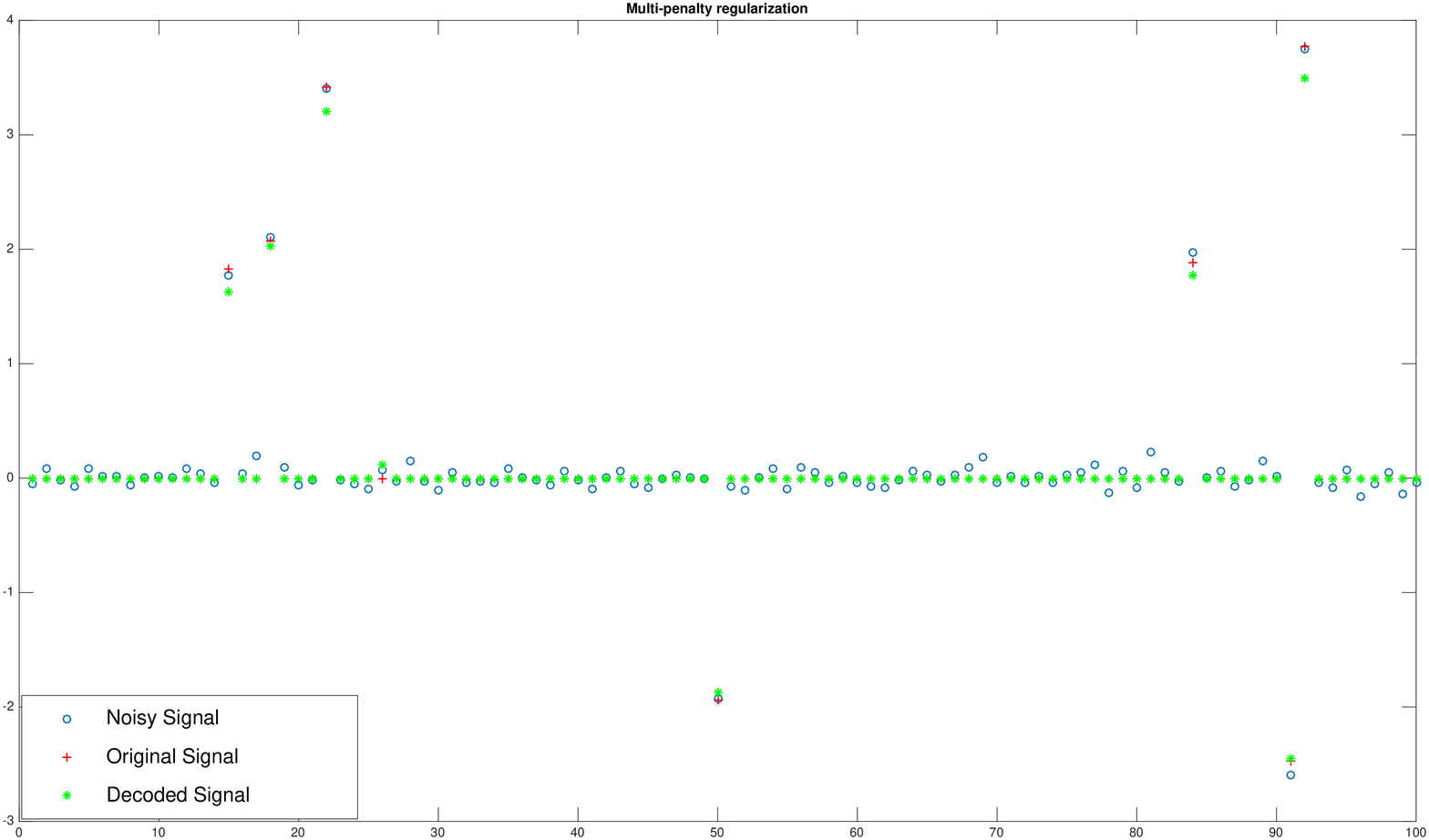}\\
  \caption{\label{fi:comp}
 The figure reports the results of two different decoding procedures
 of the same problem, where the circles represent the noisy signal and
 the crosses represent the original signal. 
 \emph{Upper figure:} results with single-penalty regularization.
 \emph{Lower figure:} results with multy-penalty regularization.
 Note that multi-penalty regularization allows for a better
 reconstruction of the support of the true signal $u^\dagger$.
  } 
\end{figure}


\begin{thebibliography}{10}

\bibitem{Arias-castro11yeldar}
E.~Arias-Castro and Y.~Eldar.
\newblock Noise folding in compressed sensing.
\newblock {\em IEEE Signal Process. Lett}, pages 478--481, 2011.

\bibitem{arfopeXX}
M.~Artina, M.~Fornasier, and S.~Peter.
\newblock Damping noise-folding and enhanced support recovery in compressed
  sensing.
\newblock {\em IEEE Trans. Signal Proc.}, 63:5990--6002, 2015.

\bibitem{BH14}
K.~Bredies and M.~Holler.
\newblock Regularization of linear inverse problems with total generalized
  variation.
\newblock {\em Journal of Inverse and Ill-posed Problems}, 1569-3945
  (online):1--38, 2014.

\bibitem{CandesTao}
E.~Candes and T.~Tao.
\newblock The dantzig selector: Statistical estimation when p is much larger
  than n.
\newblock {\em Ann. Statist.}, 35:2313--2351, 2007.

\bibitem{6204356}
M.A. Davenport, J.N. Laska, J.~Treichler, and R.G. Baraniuk.
\newblock The pros and cons of compressive sensing for wideband signal
  acquisition: Noise folding versus dynamic range.
\newblock {\em Signal Processing, IEEE Transactions on}, 60(9):4628--4642, Sept
  2012.

\bibitem{Fornasier07}
M.~Fornasier.
\newblock Domain decomposition methods for linear inverse problems with
  sparsity constraints.
\newblock {\em Inverse Problems}, 23(6):2505--2526, 2007.

\bibitem{FoRa13}
S.~Foucart and H.~Rauhut.
\newblock {\em A Mathematical Introduction to Compressive Sensing}.
\newblock Springer, New York, 2013.

\bibitem{CPA:CPA20350}
M.~Grasmair, O.~Scherzer, and M.~Haltmeier.
\newblock Necessary and sufficient conditions for linear convergence of
  l1-regularization.
\newblock {\em Communications on Pure and Applied Mathematics}, 64(2):161--182,
  2011.

\bibitem{LP_NumMath}
S.~Lu and S.~V. Pereverzev.
\newblock Multi-parameter regularization and its numerical realization.
\newblock {\em Numer. Math.}, 118(1):1--31, 2011.

\bibitem{NP13}
V.~Naumova and S.~V. Pereverzyev.
\newblock Multi-penalty regularization with a component-wise penalization.
\newblock {\em Inverse Problems}, 29(7):075002, 2013.

\bibitem{naumpeter}
V.~Naumova and S.~Peter.
\newblock Minimization of multi-penalty functionals by alternating iterative
  thresholding and optimal parameter choices.
\newblock {\em Inverse Problems}, 30:125003, 1--34, 2014.

\bibitem{SL13}
W.~Wang, S.~Lu, H~Mao, and J.~Cheng.
\newblock Multi-parameter {T}ikhonov regularization with $\ell^0$ sparsity
  constraint.
\newblock {\em Inverse Problems}, 29:065018, 2013.

\end{thebibliography}
\end{document}